\documentclass[10pt,a4paper]{article}

\usepackage[margin=30mm]{geometry}

\usepackage{indentfirst,amsmath,amssymb,latexsym,amsthm,mathrsfs,ifthen}

\newcommand{\be}{\begin{equation}}
\newcommand{\e}{\end{equation}}
\newcommand{\bea}{\begin{eqnarray}}
\newcommand{\ea}{\end{eqnarray}}

\newtheorem{theorem}{\hspace{6.5mm}Theorem}[section]
\newtheorem{proposition}[theorem]{\hspace{6.5mm}Proposition}
\newtheorem{lemma}[theorem]{\hspace{6.5mm}Lemma}
\newtheorem{corollary}[theorem]{\hspace{6.5mm}Corollary}

\newtheorem{remark}[theorem]{\hspace{6.5mm}Remark}
\newtheorem{definition}[theorem]{\hspace{6.5mm}Definition}
\newtheorem{example}[theorem]{\hspace{6.5mm}Example}

\begin{document}

\title{\large Asymptotic Exponential Arbitrage\\ and Utility-based Asymptotic Arbitrage in\\ Markovian Models of Financial 
Markets\thanks{The authors thank the referee and the associate editor for extremely constructive and
helpful reports.}}

\author{Martin Le Doux Mbele Bidima\thanks{University of Yaound\'e I, Cameroon, e-mail: mbelebidima@gmail.com} \and Mikl\'os R\'asonyi
\thanks{MTA Alfr\'ed R\'enyi Institute of Mathematics, Budapest, Hungary and University of Edinburgh, U.K., e-mail: rasonyi@renyi.mta.hu
}}

\date{\today}

\maketitle

\begin{abstract}

Consider a discrete-time infinite horizon financial market model in which the logarithm of 
the stock price is a time discretization of a stochastic differential equation. Under conditions different from those given in \cite{MR}, we prove 
the existence of investment opportunities producing an exponentially growing profit with probability tending to $1$
geometrically fast. This is achieved using ergodic results on Markov chains and tools of large deviations theory.

Furthermore, we discuss asymptotic arbitrage in the expected utility sense and its
relationship to the first part of the paper.

\end{abstract}

\noindent{\bf Keywords:} Asymptotic exponential arbitrage, Markov chains, large deviations, expected utility.

\section{Introduction}

In the classical theory of financial markets,  absence of arbitrage (riskless profit)
is characterized by the existence of suitable ``pricing rules'': risk-neutral (i.e. equivalent martingale) measures for the discounted price process of the 
risky asset. This result is often referred to as ``the fundamental theorem of asset pricing''.

Further developments of arbitrage theory encompass the so-called ``large financial markets'' (see \cite{KK2}, \cite{FSr} and the references therein).
In these papers the following common feature of numerous models is highlighted: on each finite time horizon $T>0$, there is no arbitrage opportunity but when $T$ tends to infinity, one may realize riskless profit in the long
run. Such trading opportunities
are referred to as ``asymptotic arbitrage''.

An important tool that can be used for the study of asymptotic arbitrage is the theory of large deviations (see  \cite{DZ}), as proposed in
\cite{FSr}. More recently, in \cite{MR} we presented  the discrete-time versions of some results in \cite{FSr} about asymptotic arbitrage and, in this framework, we extended
them by studying ``asymptotic exponential arbitrage with geometrically decaying probability of failure'', i.e. we discussed the possibility for investors to realize an exponentially growing profit on their long-term investments while controlling (at a geometrically decaying rate) the probability of failing to achieve such a profit. 
Some of these results were subsequently proved for continuous-time models in \cite{KN}. In the present paper
we prove results similar to Theorem 5 of \cite{MR} using different arguments and technical tools (the large deviation results of \cite{KM2} instead of those in \cite{KM1}). In this way we manage
to cover some well-known models for asset prices which were untractable in the
setting of \cite{MR}, see Examples \ref{ex2.16}, \ref{exnew} below.
We recall now the setting of \cite{MR}.

Consider a financial market in which two assets are traded: a riskless asset (a bank account or a risk-free bond) with interest rate set to $0$, i.e. with price normalized to $B_t:=1$ at all times $t\in\mathbb{N}$; and a single risky asset (such as stock) whose (discounted) price is assumed to evolve as
\be\label{eq1}
S_t:=\exp(X_t),\quad t\in\mathbb{N},
\e
where the logarithm of the stock price, $X_t$, is an $\mathbb{R}$-valued stochastic process governed by the discrete time difference equation 
\be\label{eq2}
X_t-X_{t-1}=\mu(X_{t-1})+\sigma(X_{t-1})\varepsilon_t,\quad t\geq 1,
\e
starting from a constant $X_0\in\mathbb{R}$. Here $\mu,\,\sigma:\mathbb{R}\to\mathbb{R}$ are measurable functions
(determining the drift and volatility of the stock) and $(\varepsilon_t)_{t\in\mathbb{N}}$ is an $\mathbb{R}$-valued sequence of i.i.d. random variables representing the random driving process of the stock price evolution.

Note that the log-price process $X_t$ is clearly a (discrete-time) Markov chain in the (uncountable) state space $\mathbb{R}$ (see pp. 211--228 in \cite{BW}). We suppose that its evolution is modelled on a filtered probability space $(\Omega,\mathcal{F},\mathbb{F},\mathbb{P})$, where $\mathbb{F}:=(\mathcal{F}_t)_{t\in\mathbb{N}}$ and $\mathcal{F}_t:=\sigma(X_s,0\le s\le t)$, is the natural filtration of the log-price process $X_t$ of the stock. In the sequel $\mathbb{E}$ denotes expectation with respect to the probability $\mathbb{P}$.

Trading strategies in this market are assumed $\mathbb{F}$-predictable $[0,1]$-valued processes $(\pi_t)_{t\geq 1}$ (i.e. $\pi_t$ is assumed $\mathcal{F}_{t-1}$-measurable) and no short-selling or borrowing are allowed. This means that, at each time $t$, investors allocate a proportion $\pi_t\in [0, 1]$ of their overall wealth to the 
stock while the rest remains in the bank account. Hence, given any such strategy, the corresponding wealth process $V_t^{\pi}$ of an investor  obeys the dynamics
\be\label{eq3}
 \frac{V_t^{\pi}}{V_{t-1}^{\pi}}=(1-\pi_t)+\pi_t\frac{S_t}{S_{t-1}},\quad \mbox{for all }t\ge 1,
\e
where $V_0^{\pi}:=V_0>0$ is the investor's initial capital.

\begin{definition}\label{d1}\emph{(Definitions 3 and 4 of \cite{MR})} Let $\pi_t$ be a trading strategy.

 $i)$ We say that $\pi_t$ is an asymptotic exponential arbitrage ($AEA$) in the wealth model \eqref{eq3} if there is a constant $b>0$ such that, for all $\epsilon>0$, there is a time $t_{\epsilon}\in\mathbb{N}$ satisfying
\be\label{ineq1}
\mathbb{P}(V_t^{\pi}\geq e^{bt})\geq 1-\epsilon,\mbox{ for all time}\,\,t\geq t_{\epsilon}.
\e

 $ii)$ We say that $\pi_t$ generates an asymptotic exponential arbitrage ($AEA$) with geometrically decaying probability  of failure $(GDPF)$ if there are constants $b>0$, and $c>0$ such that,
\begin{equation}\label{ineq2}
\mathbb{P}(V_t^{\pi}\geq e^{bt})\geq 1- e^{-ct}\mbox{ for all large enough }t\geq 1.
\end{equation}
\end{definition}

Clearly, $AEA$ with $GDPF$ implies $AEA$. The second kind of asymptotic arbitrage above is much  more stringent than the first one. Indeed, in \eqref{ineq1} above, there is no visible relationship between the tolerance level $\epsilon$ and the elapsed time $t_{\epsilon}$ from which the investor starts realizing exponentially growing profit; one may need to wait for a
very long time to achieve a desired tolerance level. The concept of $AEA$ with $GDPF$  removes this drawback  by allowing investors to control, at a geometrically decaying rate, the probability of failing to achieve an exponentially growing profit in the long term.

We recall the main results of \cite{MR} here. Under the following main assumptions: boundedness of the drift and volatility functions $\mu$ and $\sigma$; $\sigma$ being
bounded away from $0$ on compacts and exponential integrability of the $\varepsilon_t$s, we proved Theorem 2 (resp. Theorem 4) of \cite{MR} on the existence of $AEA$ (resp. $AEA$ with $GDPF$) in the wealth model \eqref{eq3}. 
Theorem 5 of \cite{MR} provided ergodicity-related conditions on $X_t$
which ensured AEA with GDPF.


 In sections 2 and 3 below we continue to consider the same models as in \eqref{eq1}, \eqref{eq2}, \eqref{eq3}. Under a new set of conditions on $\mu$, $\sigma$ and $(\varepsilon_t)_{t\in\mathbb{N}}$ (see $(A_1),(A_2),(A_3),(A_4)$ below), which are neither stronger nor weaker than the corresponding conditions in \cite{MR} recalled above, we show again the existence of $AEA$ with $GDPF$ (see Theorem \ref{t2.15} below), using  classical large deviations techniques from \cite{DZ}, Markov chains tools from \cite{MT} and ergodicity results on Markov chains from \cite{KM2}. Moreover, the trading strategies generating those arbitrage
opportunities will be explicitly constructed; a contribution we already obtained in \cite{MR} under different conditions, but it was absent from the inspiring continuous-time work \cite{FSr}. To get those explicit arbitrage opportunities we will be considering only stationary Markovian strategies, that is; strategies $(\pi_t)_{t\geq 1}$ where $\pi_t=\pi(X_{t-1})$, $t\ge 1$, for some fixed measurable function $\pi:\mathbb{R}\to [0,1]$.

In section 4, we will discuss the concept of ``utility-based'' asymptotic arbitrage, that is, 
asymptotic arbitrage linked to von Neumann-Morgenstern expected utilities (see Chapter 2 of \cite{FSd}). An optimal investment for 
an economic agent with utility function $U$ and time horizon $T$ is $\pi_t$ with final portfolio value $V_T^{\pi}$ for 
which the expected utility $\mathbb{E}U(V_T^{\pi})$ is maximal. We do not focus on the construction of 
optimal strategies but rather on ones that provide (rapidly) increasing expected utilities for the agent as the time horizon tends to infinity.
More precisely,
we wish to treat questions like: for power utilities $U$, and given an $AEA$ strategy $\pi_t$ as in \eqref{ineq1}, 
will the investor's expected utility $\mathbb{E}U(V_t^{\pi})$ tend to the highest available utility $U(\infty)$? If so, how fast such a convergence will take place? Conversely, if an agent pursues a trading strategy $\pi_t$  such that her/his expected utility has a convergence rate estimate, will $\pi_t$ generate $AEA$ (with $GDPF$)? We provide partial answers to these questions in Proposition \ref{p3.1}, Theorem \ref{p3.3} and Theorem \ref{t3.5} below.

\section{Main theorem on AEA with GDPF}

We denote by $\lambda$ the Lebesgue measure on $\mathcal{B}(\mathbb{R})$. 
We assume throughout this paper that the Markov chain $X_t$ satisfies the following conditions:

 $(A_1)$ The random variables $\varepsilon_t$s have a (common) density $\gamma$ with respect to $\lambda$, and this density is bounded and bounded
away from $0$ on each compact in $\mathbb{R}$.

 $(A_2)$ The drift $\mu$ is locally bounded. The volatility $\sigma$ is positive, bounded away from zero on each compact and it is (globally) bounded.

  $(A_3)$  We impose the mean-reverting drift condition 
\be\label{a1}
\limsup_{|x|\to\infty}\frac{|x+\mu(x)|}{|x|}<1. \nonumber
\e

 $(A_4)$ We assume the following integrability property for the law of the
$\varepsilon_t$s:
\be\label{a2}
 \,\exists\, \kappa>0\mbox{ such that }\mathbb{E}\big(e^{\kappa\varepsilon_1^2}\big)=:I<\infty,
\e
and $\mathbb{E}\varepsilon_1=0$ holds\footnote{This is not a restriction
of generality.
If we had $\mathbb{E}\varepsilon_1=m$, we could replace $\mu(x)$ by $\mu'(x):=\mu(x)+\sigma(x)m$
and $\varepsilon_t$ by $\varepsilon_t':=\varepsilon_t-m$ and in this way
get back to the case $\mathbb{E}\varepsilon_1=0$.}.

\begin{remark} {\rm These conditions are similar to those of \cite{MR}. The main difference
is that $\mu$ was assumed to be a bounded function in \cite{MR} while it may
be unbounded in the present paper. In this way we accomodate e.g. autoregressive processes (see Examples \ref{ex2.16} and \ref{exnew} below), which did not fit the setting
of \cite{MR}. While we relax boundedness of $\mu$, we need the integrability condition $(A_4)$ on $\varepsilon_t$, which is more stringent than the ones in \cite{MR}. Furthermore, $(A_3)$
is a much stronger ergodicity condition on the Markov chain $X_t$ than that
of Theorem 5 in \cite{MR}. Hence our main result (Theorem \ref{t2.15}) does not generalize \cite{MR} but rather complements it.}
\end{remark}

\begin{remark}
{\rm Analogously to \cite{FSr}, where the exponential of an Ornstein-Uhlenbeck process
was considered, our conditions imply that the log-price $X_t$ is ergodic
(in a strong sense). It may be argued on ecomonetric grounds that the price increments $X_t/X_{t-1}$ rather than $X_t$ should be assumed ergodic.
Just like in \cite{MR}, we opted for the present setting in order to be consistent
with \cite{FSr}. Very similar arguments could be used to prove analogous
results for the case where $X_t/X_{t-1}$ is assumed to be an ergodic Markov chain. We do not pursue this
route here.}
\end{remark}





Consider the following condition:
\be\label{rc+}
\mbox{$(RC_+)$ \,\,\,\,The set}\,\,R^+:=\{x\in\mathbb{R}\mid \mu(x)> 0\}\,\,\mbox{ satisfies }\lambda(R^+)>0.
\e

 We interpret $R^+$ as representing all states of the stock $\log$-prices $X_t$ whose 
``drift'' is positive. Thus $(RC_+)$ means that the set of states $x$ from which there is a ``bright future''
(i.e. there is an upward trend for the stock price) has positive Lebesgue measure. This is rather natural: note that
short-selling is prohibited in our model hence negative market trends
cannot be taken advantage of. We now state the main result of the present article.


\begin{theorem}\label{t2.15} Assume that $(A_1)-(A_4)$ and $(RC_+)$ hold. Then the Markovian strategy
$\pi_t^+:=\mathbf{1}_{R^+}(X_{t-1})$ produces an $AEA$ with $GDPF$.
\end{theorem}

The proof will be presented at the end of the next section, after appropriate
preparations.

\section{Large deviation estimates}

Consider the $\mathbb{R}^2$-valued auxiliary process $\Phi_t:=(X_{t-1},X_t)$, $t\ge 0$, consisting of two consecutive values of the log-price process $X_t$, where $X_{-1}$ is an arbitrarily chosen constant. We 
present below a set of preliminary results.

\begin{proposition}\label{p2.1} The process $\Phi_t$ is a Markov chain with state space $\mathbb{R}^2$.
\end{proposition}

 {\bf Proof.} We derive this from \cite{BW} pp. 211-228, where the Markov property of any (discrete-time) 
 process $Y_t$ in a Polish state space $S$ is proved when  $Y_{t+1}=m(Y_t, \xi_{t+1})$ with 
 $(\xi_t)_t$ a sequence of i.i.d. random variables independent of $Y_0$ and valued in some measurable space $S'$ and 
 $m:S\times S'\to S$ a measurable function. Clearly, $X_{t+1}=m(X_t,\varepsilon_{t+1})$ for $t\in\mathbb{N}$ 
 with $m(x,y):=x+\mu(x)+\sigma(x)y$, $x,y\in\mathbb{R}$. It follows that we have 
 $\Phi_{t+1} = (X_t,X_{t+1})=(X_t,m(X_t,\varepsilon_{t+1}))=F(\Phi_t,\xi_{t+1})$, where 
 $\xi_t:=(0,\varepsilon_t)$ and $F$ is the measurable function defined on $S\times S':=\mathbb{R}^2\times\mathbb{R}^2$ by 
 $F((x,y),(a,b)):=(y,m(y,b))$. Since the $\varepsilon_t$s are i.i.d. and independent of $X_0$, 
 the $\xi_t$s are also i.i.d. and independent from $\Phi_0$, showing the result. \hfill $_{\blacksquare}$

Notice that
$$
P(x,A):=P(X_1\in A|X_0=x)=\int_{A} p(x,y)dy,\ x\in\mathbb{R},\ A\in\mathcal{B}(\mathbb{R}),
$$
where
$$
p(x,y):=\frac{1}{\sigma(x)}\gamma\left(\frac{y-\mu(x)-x}{\sigma(x)}\right),
$$
and this function is bounded away from $0$ on each compact in $\mathbb{R}^2$,
by $(A_1)$ and $(A_2)$. 

For $z\in\mathbb{R}^2$ and $A\in\mathcal{B}(\mathbb{R}^2)$,
let $Q^t(z,A):=P(\Phi_t\in A|\Phi_0=z)$ be the $t$-step transition kernel of
the chain $\Phi_t$.

We note that, for $t\geq 2$ and $A\in\mathcal{B}(\mathbb{R}^2)$,
\begin{equation}\label{mkkk}
Q^t((u,v),A)=\int_{\mathbb{R}^{t}} 1_A(a_{t-1},a_{t})p(v,a_1)p(a_1,a_2)\ldots p(a_{t-1},a_t)da_1\ldots 
da_{t},\ u,v\in\mathbb{R}.
\end{equation}


Let $\lambda_2$ denote the Lebesgue measure on $\mathcal{B}(\mathbb{R}^2)$.

\begin{proposition}\label{p2.2} The Markov chain $\Phi_t$ is $\psi$-irreducible, i.e. there is a non-trivial measure $\psi$ such that $\psi(A)>0$ implies that for all $z$, $Q^t(z,A)>0$ for some $t$.
\end{proposition}

 {\bf Proof.} It suffices to check that $\lambda_2$ is such a measure. If $\lambda_2(A)>0$ then for $t=2$ we
get from \eqref{mkkk} that 
$$
Q^2((u,v),A)=\int_{\mathbb{R}^{2}} 1_A(a_1,a_2)p(v,a_1)p(a_1,a_2)da_1\,da_2>0
$$
since $p(v,a_1)p(a_1,a_2)$ is (everywhere) positive.
\hfill $_{\blacksquare}$

We recall two definitions from Chapter 5 of \cite{MT} in our specific setting. 
A set $C_2\subset \mathbb{R}^2$ is called \emph{small} if
$$
Q^t(x,A)\ge\mu(A)\,\,\mbox{for all}\,\, x\in C_2,\,A\in\mathcal{B}(\mathbb{R}^2)
$$
with some non-trivial measure $\mu$. The chain $\Phi_t$ is \emph{aperiodic} if, 
for some small set $C_2$ and corresponding measure $\mu$, the greatest common
divisor of the set
\[
E_{C_2}:=\{n\ge 1: \mbox{for all }x\in C_2,\ Q^n(x,A)\ge\delta_n\mu(A)\mbox{ for some $\delta_n>0$}\},
\]
is $1$.

\begin{proposition}\label{p2.3}
If $C$ is a compact interval in $\mathbb{R}$ then $C_2:=\mathbb{R}\times C$ is a 
small set for the Markov chain $\Phi_t$, and this chain is aperiodic.
\end{proposition}

{\bf Proof.} It suffices to show that for all $u\in\mathbb{R}$ and $v\in C$,
$$
Q^i((u,v),A)\geq c_i\lambda_2(A\cap (C\times C))
$$
for $i=2,3$ and appropriate constants $c_2,c_3>0$ since this implies $2,3\in E_{C_2}$.
This is true by \eqref{mkkk} with
$$
c_2=\inf_{v, a_1\in C}p(v,a_1)\inf_{a_1,a_2\in C}p(a_1,a_2),\quad  
c_3=\inf_{a_2,a_3\in C}p(a_2,a_3) \inf_{v,a_1\in C}p(v,a_1)
\inf_{a_1,a_2\in C}p(a_1,a_2)\lambda(C).
$$
\hfill $_{\blacksquare}$

Now we need certain moment estimates.
\begin{lemma}\label{l2.4} The random variable $\varepsilon_1$ in \eqref{a2} of Assumption $(A_4)$ satisfies the following property: there is $c>0$ such that for every real number $a\geq 1$ we have
\begin{equation}\label{csillagok}
\mathbb{E}\big(e^{a|\varepsilon|}\big)\le e^{ca^2}.
\end{equation}
\end{lemma}

{\bf Proof.} Set $\xi:=|\varepsilon_1|$. Then we have
 \be
 \begin{array}{lll}
  \mathbb{P}\big(e^{a\xi}>x\big) &=& \mathbb{P}\Big(\exp\Big(\kappa\Big[\frac{\log(e^{a\xi})}{a}\Big]^2\Big)>\exp\Big(\kappa\Big[\frac{\log x}{a}\Big]^2\Big)\Big)\\ \nonumber

                           &\le& I\exp\Big(-\kappa\big(\log(x)/a\big)^2\Big)\,\,\mbox{by Markov's inequality} \\ \nonumber
                           &=& I(\frac{1}{x})^{(\kappa/a^2)\log x},
\end{array}
\e
see \eqref{a2} for the definition of $I$.
 Since the exponent $(\kappa/a^2)\log x > 2$ provided that $x>e^{2a^2/\kappa}$, we have $\mathbb{E}\big(e^{a\xi}\big)=\int_0^{\infty}\mathbb{P}\big(e^{a\xi}>x\big)dx\le e^{2a^2/\kappa}+I\int_{\exp(2a^2/\kappa)}^{\infty}1/x^2dx$. The last integral is less than $\int_1^{\infty}1/x^2dx$, which is finite, thus we conclude the proof of
\eqref{csillagok} by taking $c=c_1+(2/\kappa)$ with $c_1>0$ large enough. \hfill $_{\blacksquare}$

The proof of Theorem \ref{t2.15} will be based  on results from \cite{KM2}. In order to apply
the results of that paper we will need to verify that the Markov chain $\Phi_t$ satisfies condition $(DV3+)$
below. We formulate this condition only in the case where the state space is $\mathbb{R}^d$.


We say that a $\psi$-irreducible and aperiodic Markov chain $Z_t$ with transition law 
$R=R(x,A)$ satisfies condition $(DV3+)$ if
\begin{enumerate}

\item[(i)] There are measurable functions $V,W:\mathbb{R}^d\to [1,\infty)$ and a small
set $C$ such that for all $x\in\mathbb{R}^d$,
\begin{equation}\label{cry}
\log (e^{-V}Re^V)(x)\leq -\delta W(x) + b\mathbf{1}_C(x)
\end{equation}
for some $\delta,b>0$.

\item[(ii)] There exists $t_0>0$ such that, for each $r<\Vert W\Vert_{\infty}$, there is
a measure $\beta_r$ with $\beta_r(e^V)<\infty$ and 
\begin{equation}\label{ad}
\mathbb{P}_x(Z_{t_0}\in A\mbox{ and }Z_t\mbox{ has not quitted }C_W(r)\mbox{ before }t_0+1)\leq
\beta_r(A)
\end{equation}
for all $x\in C_W(r)$ and $A\in\mathcal{B}(\mathbb{R}^d)$,
where $C_W(r)=\{y\in\mathbb{R}^d:\ W(y)\leq r\}$.
\end{enumerate}

We now recall the results of \cite{KM2} which we will need in the sequel. Let $W_0:\mathbb{R}^d\to [1,\infty)$
such that
\be\label{eq18}
\lim_{r\to\infty}\sup_{x\in\mathbb{R}^d}\Big(\frac{W_0(x)}{W(x)}{\bf 1}_{\{W(x)>r\}}\Big)=0.
\e
Next, consider the Banach space $L_{\infty}^{W_0}:=\{g:\mathbb{R}^d\to\mathbb{C}:\sup_{x}\frac{|g(x)|}{W_0(x)}<\infty\}$, 
equipped with the norm $\|g\|_{W_0}:=\sup_{x}|g(x)|/W_0(x)$.

\begin{theorem}\label{alap} Let $Z_t$ satisfy $(DV3+)$ with unbounded $W$. Then
$Z_t$ admits an invariant probability measure $\nu$, the limit
$$
\Lambda(g):=\lim_{t\to\infty} \frac{1}{t}\ln \mathbb{E}_z [\exp(\sum_{n=1}^t g(Z_n))]
$$
exists and it is finite for all $g\in L_{\infty}^{W_0}$ and for all initial values $Z_0=z$ (and it is independent of $z$). Fix $g_0\in L_{\infty}^{W_0}$.
The function
$\theta\to \Lambda(g_0+\theta g)$ is analytic in $\theta$ with Taylor-expansion
$$
\Lambda(g_0+\theta g)=\Lambda(g_0)+\theta \nu(g)+\frac{1}{2}\theta^2 \sigma^2(g)+O(\theta^3),
$$
where $\sigma^2(g):=\lim_{t\to\infty}(1/t)\mathrm{var}(g(Z_0)+\ldots+g(Z_{t-1}))$.
\end{theorem}
\begin{proof}
This follows from Theorems 1.2 and 4.3 of \cite{KM2}.
\end{proof}

Let us define $\Lambda_g(\theta):=\Lambda (\theta g)$ for $\theta\in\mathbb{R}$. Denote
$$
\Lambda_g^*(x):=\sup_{\theta\in\mathbb{R}}(\theta x-\Lambda_g(\theta)),\quad x\in\mathbb{R},
$$
the Fenchel-Legendre conjugate of $\Lambda_g(\cdot)$.
\begin{corollary}\label{kp}
Under the conditions of the previous Theorem, if $\sigma^2(g)>0$ then
$\Lambda_g^*(x)>0$ for all $x\neq \nu(g)$.
\end{corollary}

{\bf Proof.} $\Lambda_g$ is analytic, a fortiori, it is differentiable. $\Lambda_g(0)=0$ by the definition of $\Lambda$. 
From the Taylor expansion of the preceding Theorem, 
$\Lambda_g'(0)=\nu(g)$ and $\Lambda_g''(0)=\sigma^2(g)>0$ so we get that $\Lambda_g^*(\nu(g))=\nu(g)\times 0-\Lambda_g(0)=0$. By the definition of a conjugate function we always have 
$\Lambda_g^*(x)\ge 0\times x-\Lambda_g(0)=0$ for all $x\in\mathbb{R}$. It follows that $\nu(g)$ is a global minimiser for $\Lambda_g^*$. 
By the differentiability of $\Lambda_g$, $\Lambda_g^*$ is strictly convex on its effective domain. This implies that the global minimiser 
$\nu(g)$ for $\Lambda_g^*$ is unique. This uniqueness implies that $\Lambda_g^*(x)>0$ for all $x\ne\nu(g)$. 
\hfill $_{\blacksquare}$

In order to apply these results to our long-term investment problems we need to establish that $\Phi_t$ satisfies $(DV3+)$.
First we prove a related statement about $X_t$.

\begin{proposition}\label{p2.5} The Markov chain $X_t$ satisfies the ``drift condition'' \eqref{cry} for $d=1$,
$R(x,A)=P(x,A)$ with a suitable compact interval $C\subset \mathbb{R}$ and $V(x)=W(x)=1+qx^2$ with a suitable $q>0$.
\end{proposition}

 {\bf Proof.} Recall that $R\,e^V(x):=\int e^{V(y)}R(x,dy)$, for all $x\in\mathbb{R}$.
We have to show
\be\label{ineq10}
 P\,e^V(x)\le e^{V(x)-\delta W(x) + b{\bf 1}_C(x)}\,\,\mbox{for all}\,\, x\in\mathbb{R}
\e
for suitably small $q>0$ and $C=[-K,K]$ with $K$ suitably large.

 Since $Pe^V(x)=\mathbb{E}\big(e^{V(X_1)}\mid X_0=x\big)=\mathbb{E}\big(e^{V(x+\mu(x)+\sigma(x)\varepsilon_1)}\big)$, it follows from \eqref{ineq10} that we need to show,
\be\label{ineq11}
 \mathbb{E}\big(e^{1+q(x+\mu(x))^2+2q(x+\mu(x))\sigma(x)\varepsilon_1+q\sigma^2(x)\varepsilon_1^2}\big)\le e^{(1-\delta)V(x)+b{\bf 1}_C(x)}\,\,\mbox{for all}\,\,x\in\mathbb{R}.
\e

 To get this, it is sufficient to prove the two claims below:

 {\bf Claim 1:} For all $x$ with $|x|>K$ with $K$ large enough we have
\be\label{e8}
 \mathbb{E}\big(e^{1+q(x+\mu(x))^2+2q(x+\mu(x))\sigma(x)\varepsilon_1+q\sigma^2(x)\varepsilon_1^2}\big)\le e^{(1-\delta)(1+qx^2)}
\e

 {\bf Claim 2:} For ``small'' $x$ (i.e. $|x|\leq K$ we have
\be\label{e9}
 \sup_{x\in C}\mathbb{E}\big(e^{1+q(x+\mu(x))^2+2q(x+\mu(x))\sigma(x)\varepsilon_1+q\sigma^2(x)\varepsilon_1^2}\big)<G(K),
\e
for some positive constant $G(K)$.

 {\bf Proof of Claim 1.} Using Assumption $(A_3)$, for $|x|$ large enough, there is a small $\delta>0$ such 
 that $(x+\mu(x))^2\le (1-4\delta)x^2$. Since $1\le\delta(1+qx^2)$ for $|x|$ large enough, it follows that 
 $e^{1+q(x+\mu(x))^2}\le e^{(1-3\delta)(1+qx^2)}$.

By $(A_2)$ there is $M>0$ such that, for all $x$, $\sigma(x)\leq M$. 
If we choose $q$ such that $q M^2<\kappa/2$  then it is enough to show that
$
 \mathbb{E}\big(e^{2q|x+\mu(x)|M|\varepsilon_1|+(\kappa/2)\varepsilon_1^2}\big)\le e^{2\delta qx^2}
$.
By the Cauchy-Schwarz inequality, it suffices to prove
\be\label{e10}
 \sqrt{\mathbb{E}\big(e^{4q|x+\mu(x)|M|\varepsilon_1|}\big)}\sqrt{\mathbb{E}\big(e^{\kappa\varepsilon_1^2}\big)}\le e^{2\delta qx^2}
\e

 By \eqref{a2}, the second term on the left-hand side of \eqref{e10} is the constant $\sqrt{I}$. This is smaller than $e^{\delta qx^2}$ for large enough $|x|$. So, since again by $(A_3)$, $4q|x+\mu(x)|M\le 4qM|x|$ for $|x|$ large enough, it remains to show
$
 \sqrt{\mathbb{E}\big(e^{4qM|x||\varepsilon_1|}\big)}\le e^{\delta qx^2}\,\,\mbox{for large}\,\,|x|
$, or, equivalently,
\be\label{e11}
 \mathbb{E}\big(e^{4qM|x||\varepsilon_1|}\big)\le e^{2\delta qx^2}\,\,\mbox{for large}\,\,|x|.
\e
Applying Lemma \ref{l2.4}, the left-hand side of \eqref{e11} is smaller than $e^{16cq^2M^2|x|^2}$ for some fixed constant $c>0$. Hence, if one chooses $q$ small enough such that $16q^2M^2c<2\delta q$ and $qM^2<\kappa/2$ then \eqref{e11} holds, showing Claim 1.

 {\bf Proof of Claim 2.} By Assumption $(A_2)$, $\mu$ is bounded above on any compact 
 $C=[-K,K]$ by some positive constant $A=A(K)$ and the function 
 $x\mapsto (x+\mu(x))^2$ is also bounded on $C$ by some positive constant $B=B(K)$. 
 Applying Cauchy-Schwarz Inequality and \eqref{a2},
\[
 \begin{array}{lll}
 \mathbb{E}\big(e^{1+q(x+\mu(x))^2+2q(x+\mu(x))\sigma(x)\varepsilon_1+q\sigma^2(x)\varepsilon_1^2}\big) &\le& 
 \mathbb{E}\big(e^{1+qB+2q(K+A)M|\varepsilon_1| +(\kappa/2)\varepsilon_1^2}\big)\\
       &\le& e^{(1+qB)}\sqrt{\mathbb{E}\big(e^{4q(K+A)M|\varepsilon_1|}\big)}\sqrt{\mathbb{E}\big(e^{\kappa\varepsilon_1^2}\big)}\\
       &=& e^{(1+qB)}\sqrt{I}\sqrt{\mathbb{E}\big(e^{4q(K+A)M|\varepsilon_1|}\big)}
\end{array}
\]

 We then choose $K$ large enough such that $4q(K+A)M\geq 1$ and we get, by Lemma \ref{l2.4}, that for all $x\in C=[-K,K]$,
 \[
\mathbb{E}\big(e^{1+q(x+\mu(x))^2+2q(x+\mu(x))\sigma(x)\varepsilon+q\sigma^2(x)\varepsilon^2}\big)\le e^{(1+qB)}\sqrt{I}\sqrt{e^{16c'q^2(K+A)^2M^2}},
\]
for a fixed constant $c'>0$. This holds for all $x\in C$, hence \eqref{e9} holds true when taking the supremum over $C$ of the left-hand side of this latter inequality. 
\hfill $_{\blacksquare}$

\begin{proposition}\label{p2.555} The Markov chain $\Phi_t$ satisfies $(DV3+)\,(i)$.
\end{proposition}

{\bf Proof.} We follow Proposition 4.1 of \cite{KM2} and deduce this statement from 
Proposition \ref{p2.5} above.  Recall
$V(x)=W(x)=1+qx^2$, $C$ and $\delta>0$ from that Proposition.
Take $C_2:=\mathbb{R}\times C$. For $x,y\in\mathbb{R}$
define $V_2(x,y):=V(y)+(\delta/2) W(x)$ and $W_2(x,y):=(1/2)(W(x)+W(y))$.
Then
\begin{eqnarray*}
\log e^{-V_2}Qe^{V_2}(x,y)&=&-V(y)-(\delta/2) W(x) +\log \int_{\mathbb{R}}e^{\frac{\delta}{2}W(y)
+V(z)} P(y,dz)\\ &\leq& -V(y)-(\delta/2) W(x) +(\delta/2) W(y) + [V(y)-\delta W(y)+
b 1_C (y)]\\ &\leq& -\delta W_2(x,y) + b 1_{C_2}(x,y),
\end{eqnarray*}
showing that \eqref{cry} is true with $V_2,W_2$. As $C_2$ has been shown to be small in Proposition \ref{p2.3},
we conclude.
\hfill $_{\blacksquare}$

\begin{proposition}\label{p2.9} The chain $\Phi_t$ satisfies condition $(DV3+)\,(ii)$ as well.
\end{proposition}

 {\bf Proof.}  Consider $V_2(x,y),W_2(x,y)$, defined in the previous
Proposition. We choose $t_0:=2$, and let $r<\|W\|_{\infty}=\infty$.
 
It suffices to prove existence of a measure $\beta_r$ on $\mathcal{B}(\mathbb{R}^2)$ such that,
\be\label{e26}
 \beta_r(e^{V_2})<\infty\,\,\mbox{and}\,\,Q^2\big((x,y), D\cap C_W(r) \big)\le\beta_r(D),
\e
for all $(x,y)\in C_W(r)$ and all $D\in\mathcal{B}(\mathbb{R}^2)$.

Let $H$ denote the projection of $C_W(r)$ on the first coordinate (which is the same as its projection on
the second coordinate). By $(A_1)$ and $(A_2)$, the function $p(x,y)$ is bounded on $H\times H$ by a constant $J$.
Hence
$$
Q^2\big((x,y), D\cap C_W(r) \big)\le \int_{D\cap C_W(r)} p(y,a_1)p(a_1,a_2)da_1 da_2\leq J^2\lambda_2(D\cap C_W(r))=:
\beta_r(D).
$$
Finally, it is clear that $\beta_r(e^{V_2})<\infty$ as it
is the Lebesgue-integral of a continuous function on a compact of $\mathbb{R}^2$. 
\hfill $_{\blacksquare}$

\begin{corollary}\label{c2.7} The Markov chain $\Phi_t$ has an invariant probability measure $\nu$ equivalent to $\lambda_2$.
\end{corollary}

 {\bf Proof.} Theorem \ref{alap} implies that $\Phi_t$ has an invariant probability measure, say, $\nu$.

 Furthermore, from \eqref{mkkk}, $\mathbb{P}(\Phi_2\in\cdot\vert \Phi_0=(x,y))$ is $\lambda_2$-absolutely continuous 
 for each $(x,y)\in\mathbb{R}^2$, hence we get $\nu \ll \lambda_2$. On the other hand, 
 from the definitions of recurrent and positive chains on pages 186 and 235 of \cite{MT}, 
 it follows by Proposition 10.1.1 and Theorem 10.4.9 of the same reference that $\nu\sim\psi$,
 where $\psi$ is a maximal irreducibility measure. Hence $\psi \gg \lambda_2$ by Proposition 4.2.2 $(ii)$ in \cite{MT}, so
 we get $\nu \gg \lambda_2$. It follows that $\nu\sim\lambda_2$, as required. \hfill\  $_{\blacksquare}$

 We now proceed to a proper investigation of asymptotic arbitrage exponential opportunities in the wealth model \eqref{eq3}. Inspecting again the dynamics of the investor's wealth process $V_t^{\pi}$ in this model, for any Markovian strategy $\pi_t$, we may express it in the form
\be\label{e32}
  V_t^{\pi} = V_0\exp\big(\sum_{n=1}^tf(\Phi_n)\big) = V_0\exp\Big(t\frac{\sum_{n=1}^tf(\Phi_n)}{t}\Big),\,\,\mbox{for all}\,\,t\ge 1,
\e
where the function $f$ is defined by 
\begin{equation}\label{suttog}
f(x,y):= \log\big((1-\pi(x))+\pi(x)\exp(y-x)\big),\quad x,y\in\mathbb{R},
\end{equation}
and $\Phi_t=(X_{t-1},X_t)$, $t\in\mathbb{N}$, is the Markov chain
in consideration. We will need to insure that, for any Markovian strategy $\pi_t$, the sequence of random variables $\log(V_t^{\pi}/V_0)=\sum_{n=1}^tf(\Phi_n)$ satisfies a large deviation principle ($LDP$) hypotheses. That is, we will need that the limit $\Lambda_f(\theta):=\lim_{t\to\infty}\frac{1}{t}\log\mathbb{E}(e^{\theta\sum_{n=1}^t f(\Phi_n)})$ exists, for each $\theta\in\mathbb{R}$, with $\Lambda_f$ satisfying the remaining conditions in G\"artner-Ellis Theorem as stated in Theorem 2.3.6 in \cite{DZ}.

Define the function $W_0:\mathbb{R}^2\to [1,\infty)$ 
by $W_0(x,y):=1+|x|+|y|$, for all $x,y\in\mathbb{R}$. Clearly, $W_0$ satisfies \eqref{eq18} with $d=2$ and $W=W_2$.


\begin{lemma}\label{l2.10} The function $f$ belongs to the space $L_{\infty}^{W_0}$.
\end{lemma}

 {\bf Proof.} For all $x,y\in\mathbb{R}$, since $\pi(x)\in\ [0,1]$, we have $1-\pi(x)+\pi(x)\exp(y-x)\leq 1+\exp(y-x)$. It follows that $f(x,y)\leq \vert x\vert+\vert y\vert +1$.

 On the other hand, for $0\leq a\leq 1/2$, we have $1-a+a \exp(y-x)\geq 1/2$. And for $a>1/2$, we have $1-a+a \exp(y-x)\geq (1/2)\exp(y-x)$. To sum up, we obtain that $f(x,y)\geq \log (1/2)-\vert x\vert -\vert y\vert$ for all $x,y\in\mathbb{R}$.

 Hence $\vert f(x,y)\vert\leq c(1 +\vert x\vert +\vert y\vert)$, for some constant $c>0$, and the claim follows. \hfill $_{\blacksquare}$

\begin{proposition}\label{p2.11} Let $\pi_t$ be any Markovian strategy in the wealth model \eqref{eq3}. Then 
$\Lambda_f(\theta):=\lim_{t\to\infty}\frac{1}{t}\log\mathbb{E}_{(X_{-1},X_0)}\big(e^{\theta\sum_{n=1}^tf(\Phi_n)}\big)$, $\theta\in\mathbb{R}$ is a well-defined analytic function so the averages $\frac{1}{t}\log(V_t^{\pi}/V_0)=\frac{1}{t}\sum_{n=1}^tf(\Phi_n)$ satisfy a large deviations estimate with good rate function $\Lambda_f^*$ (the convex conjugate of $\Lambda_f$).
\end{proposition}

 {\bf Proof.} By Theorem \ref{alap}, $\Lambda_f$ verifies the conditions of the G\"artner-Ellis Theorem 2.3.6 in \cite{DZ} 
 (analyticity implies essential smoothness). Applying this theorem we conclude.
\hfill $_{\blacksquare}$

\begin{lemma}\label{l2.14} If $(RC_+)$ is satisfied then the Markovian strategy $\pi^+(x):=\mathbf{1}_{R^+}(x)$
is such that
\be
\nu(f)=\mathbb{E}\Big(\log\big((1-\pi^+(\tilde{X_0}))+\pi^+(\tilde{X_0})\exp(\tilde{X_1}-\tilde{X_0})\big)\Big)>0,
\e
where the pair of random variables $(\tilde{X_0},\tilde{X_1})$ has distribution $\nu$.
\end{lemma}

 {\bf Proof.} Since $\nu$ is a probability measure on $\mathcal{B}(\mathbb{R}^2)$ and is invariant for the chain $\Phi_t=(X_t,X_{t+1})$, there is a pair of $\mathbb{R}$-valued random variables $(\tilde{X_0},\varepsilon_1)$ such that $\varepsilon_1$ is independent of $\tilde{X_0}$ and, defining 
 $\tilde{X_1}=\tilde{X_0}+\mu(\tilde{X_0})+\sigma(\tilde{X_0})\varepsilon_1$, the pair $(\tilde{X_0},\tilde{X_1})$ has distribution $\nu$. For all $x\in\mathbb{R}$,
\be
 \begin{array}{lllll}
 \mathbb{E}\big(\tilde{X_1}\mid \tilde{X_0}=x\big) &=& \mathbb{E}\big(x+\mu(x)+\sigma(x)\varepsilon_1\mid \tilde{X_0}=x\big)\\
                                  &=& x+\mu(x)+\sigma(x)\mathbb{E}(\varepsilon_1\mid \tilde{X_0}=x)\\
                                   &=& x+\mu(x)+\sigma(x)\mathbb{E}(\varepsilon_1)\\
                                  &=& x+\mu(x), \nonumber\\

\end{array}
\e
since $\varepsilon_1$ independent of $\tilde{X}_0$. It follows that if $x\in R_+$ then we have
\begin{equation}\label{mkk}
\mathbb{E}(\tilde{X_1}\mid \tilde{X_0}=x)>x.
\end{equation}

 Consider now the explicitly defined Markovian strategy $\pi^+(x):={\bf 1}_{R^+}(x)$ for $x\in\mathbb{R}$, which is constructed as follows: at each time $t>0$, we invest all the current wealth in the stock if the $\log$-market price of risk is above $0$, and we put everything in the bank account otherwise. Given this strategy, consider the corresponding $f(x,y)= \log\big((1-\pi^+(x))+\pi^+(x)\exp(y-x)\big)$, $(x,y)\in\mathbb{R}^2$. Since by definition $\nu(f)=\int_{\mathbb{R}^2}f(x,y)\nu(dx,dy)$, we have $\nu(f)=\mathbb{E}\Big(\log\big((1-\pi^+(\tilde{X_0}))+\pi^+(\tilde{X_0})\exp(\tilde{X_1}-\tilde{X_0})\big)\Big)$. Next, by Corollary \ref{c2.7}, $\nu$ has a $\lambda_2$-a.e. positive density with respect to $\lambda_2$, hence its $\tilde{X}_0$-marginal, denoted by $\eta$, has a $\lambda$-a.e. positive density $\ell(x)$. Therefore we obtain that
\[
\begin{array}{llll}
 \nu(f) &=& \int_{\mathbb{R}}\mathbb{E}\big(\log((1-\pi^+(x))+\pi^+(x)\exp(\tilde{X_1}-x))\mid \tilde{X_0}=x\big)\eta(dx) \\
        &\ge& \int_{{R}^+}\mathbb{E}\big(\log\exp(\tilde{X_1}-x)\mid \tilde{X_0}=x\big)\eta(dx) \\
        &=& \int_{{R}^+}\mathbb{E}\big(\tilde{X_1}-x\mid \tilde{X_0}=x\big)\ell(x)\lambda(dx)\\
        &>& 0,
\end{array}
\]
by \eqref{mkk}, showing the lemma. \hfill $_{\blacksquare}$


 {\bf Proof of Theorem \ref{t2.15}.} If $f$ is $\nu$-a.s. constant then the statement is trivial. If not, then $\sigma^2(f)>0$
by the argument of  Theorem 5 in \cite{MR}. Proposition \ref{p2.11} says that $\frac{1}{t}\log(V_t^{\pi^+}/V_0)$ satisfies 
a large deviations principle with good rate function $\Lambda_f^*$.
In particular,
applying the upper large deviations inequality $(2.3.7)$ of Theorem 2.3.6 in \cite{DZ} we get
\be\label{ineq25}
 \limsup_{t\to\infty}\frac{1}{t}\log\mathbb{P}\Big(\frac{1}{t}\log(V_t^{\pi^+}/V_0)<\nu(f)/2\Big)\le 
 -\inf_{x\in (-\infty,\nu(f)/2]}\Lambda_f^*(x),
\e
where $\nu(f)>0$ by Lemma \ref{l2.14}. By Corollary \ref{kp}, $\Lambda_f^*(\nu(f)/2)>0$ and $\nu(f)$ is the unique minimiser of $\Lambda_f^*$. By strict convexity, $\Lambda_f^*$ is decreasing on $(-\infty,\nu(f)]$. These imply that the right hand side of \eqref{ineq25} is equal to $-\Lambda_f^*(\nu(f)/2)$. Hence
\be\label{e38}
 \mathbb{P}\big(V_t^{\pi^+}\ge e^{\log(V_0)+\nu(f)t/2}\big)\ge 1-e^{-t\Lambda_f^*(\nu(f)/2)}\,\,\mbox{for large}\,\,t,
\e
and the result follows. \hfill $_{\blacksquare}$

The following two examples fall outside the scope of \cite{MR} but can be treated using Theorem \ref{t2.15} above.

\begin{example}\label{ex2.16} \textrm{\bf Stable autoregressive process.}\hfill

 \emph{Consider the model $S_t:=e^{X_t}$ where $X_t$ is a stable
autoregressive process (that is, the  discrete-time version of the Ornstein-Uhlenbeck process):
\be\label{e2.40}
X_{t+1}=\alpha X_t+\varepsilon_{t+1},\,\,\mbox{for all}\,\,t\ge 1,
\e
where $0<\vert \alpha\vert <1$, $X_0$ is constant and the $\varepsilon_t$ are i.i.d. $\mathcal{N}(0,1)$.}

 \emph{In this typical example, the drift and volatility functions are identified as $\mu(x)=(\alpha-1)x$ and $\sigma(x)=1$, for all $x\in\mathbb{R}$. All the assumptions $(A_1)-(A_4)$ on $\mu$, $\sigma$ and on the $\varepsilon_t$s trivially hold.}

 \emph{Next, $R^+=(-\infty,0)$. Obviously, $(RC_+)$ holds.  
 It follows that the trading opportunity $\pi_t^+=\pi^+(X_{t-1})$, with $\pi^+:={\bf 1}_{R^+}$ realizes an $AEA$ with $GDPF$, by Theorem \ref{t2.15}.} \hfill $_{\blacksquare}$
\end{example}

\begin{example}\label{exnew} \emph{\bf A Cox-Ingersoll-Ross--type process.}\hfill

{\rm In mathematical finance the process $H_t$ described by the stochastic
differential equation
\begin{equation}\label{cir}
dH_t=-\beta H_t dt+ \sigma\sqrt{|H_t|}dW_t
\end{equation}
is often used to model stochastic volatility or the short rate in bond markets,
here $W_t$ is a Brownian motion.
We present here a slight modification of the discretization of this model.
The modifications are necessary, since the volatility of $H_t$ is neither
bounded above nor bounded away from $0$.

Let us define the log-price process by
$$
X_{t+1}=\alpha X_t +\sigma \min\{\max\{\sqrt{|X_t|},c_1\},c_2\}\varepsilon_t,
\quad t\geq 1,
$$
where $|\alpha|<1$, $\sigma>0$, $0<c_1<c_2$ are given constants and $\varepsilon_t$
is $\mathcal{N}(0,1)$. It is easy to check that this process
also satisfies the conditions of Theorem \ref{t2.15}.}

\end{example}

\section{Utility-based Asymptotic Arbitrage}


 We consider risk-averse investors with initial capital $V_0=x\in (0,\infty)$ who express their preferences
in terms of a utility function $U:(0,\infty)\to\mathbb{R}$, where $U$ belongs to the subclass of Hyperbolic Absolute Risk Aversion ($HARA$) 
utility functions $U(x)=x^{\alpha}$, with $0<\alpha< 1$, or $U(x)=-x^{\alpha}$, with $\alpha< 0$, for all $x\in (0,\infty)$. 
The parameter $\alpha$ is related to risk-aversion: the larger $-\alpha$ is, the more afraid investors become of losses, see \cite{FSd}.

 As mentioned in the introductory section, in this paper we do not intend to solve the finite horizon utility maximization problem which is well-discussed in the 
 literature  and which consists in finding the maximal expected utility $u(x):=\sup_{\pi}\mathbb{E}U(V_T^{\pi})$ together with an optimal strategy $(\pi^*_t)_{1\le t\le T}$ verifying $u(x)=\mathbb{E}U(V_T^{\pi^*})$. Instead, we focus on trading opportunities that provide (rapidly) increasing expected utilities for the agents as the time horizon tends to infinity, in the
spirit of \cite{Dk} and \cite{FSr}. 

Consider first the subclass of power utility functions $U(x):=x^{\alpha}$, with $0<\alpha<1$, for $x\in (0,\infty)$. 

\begin{proposition}\label{p3.1} If a trading strategy $\pi_t$ realizes an $AEA$  then there is a constant $b>0$ such that
\be
 \mathbb{E}U(V_t^{\pi})\ge e^{\alpha bt},\,\,\mbox{for all large enough }\,\, t.
\e
\end{proposition}

 {\bf Proof.} By definition of $AEA$, there are a constant $b>0$ and a time $t_{1/2}$ such that
$\mathbb{P}(V_t^{\pi}\ge e^{bt})\ge 1/2$ for all time $t\geq t_{1/2}$. It follows that 
$\mathbb{E}U(V_t^{\pi})\ge\mathbb{E}U(e^{bt}){\bf 1}_{\{V_t^{\pi}\ge e^{bt}\}}\ge (1/2)e^{\alpha bt}\ge e^{\alpha b't}$ for any $b'<b$ and for all $t$ large enough, as required. \hfill $_{\blacksquare}$




Next, suppose that investors choose from the second subclass of power utility functions $U(x):=-x^{\alpha}$, with $\alpha<0$, for all $x\in (0,\infty)$. These functions express larger risk-aversion and are thought to be more realistic. We derive the key result of this section below. We remark that, despite the short proof, the following theorem relies on all the heavy machinery of the paper \cite{KM2} as well as on our preliminary results established in Section 2 and it is, in fact, highly non-trivial.

\begin{theorem}\label{p3.3} Assume that the log-price $X_t$ satisfies $(A_1)-(A_4)$ as well as $(RC_+)$. Let $\pi^+_t$ be the strategy defined in 
Theorem \ref{t2.15}. Then there is $\alpha_0<0$ such that for any risk-aversion coefficient $0>\alpha>\alpha_0$, the expected utility of the corresponding investor's wealth converges to $0$ at an exponential rate. That is, with the power utility $U(x):=-x^{\alpha}$, $x\in (0,\infty)$, we have,
\be\label{e4.4}
|\mathbb{E}U(V_t^{\pi^+})|\le  K e^{-ct},\,\,\mbox{for all large enough}\,\, t,
\e
for some constants $K=K(\alpha),c=c(\alpha)>0$.
\end{theorem}

 {\bf Proof.} Recall section 2, in particular, Corollary \ref{c2.7}, Proposition \ref{p2.11} and \eqref{suttog}. When $f$ is constant $\nu$-a.s., the statement is trivial. Otherwise we may assume $\sigma^2(f)>0$
(see the proof of Theorem \ref{t2.15}). In section 2 we obtained  that $\Lambda_f(0)=0$, $\Lambda_f'(0)=\nu(f)>0$. Since by analyticity of $\Lambda_f$, $\Lambda_f'$ is continuous, there exists $\alpha_0<0$ such that $\Lambda_f(\alpha)<0$ for $\alpha\in (\alpha_0,0)$. Theorem 3.1 of \cite{KM2} implies that for some constant $d_{\alpha}$, we have
\be\label{pope}
\frac{-\mathbb{E}e^{\alpha (f(\Phi_1)+\ldots+f(\Phi_n))}}{e^{n\Lambda_f(\alpha)}}=
\frac{ \mathbb{E} U(V^{\pi}_n)/V_0^{\alpha}}
{e^{n\Lambda_f(\alpha)}}
\to d_{\alpha},\ n\to\infty,
\e
showing the statement. \hfill $_{\blacksquare}$

 It seems that, in general, we should not expect more than this result (i.e. we cannot get such a 
theorem for all $\alpha<0$). To illustrate this, we now construct an example where there is 
$AEA$ with $GDPF$ but, for some $\alpha<0$, we have $\mathbb{E}(V_t^{\pi})^{\alpha}\to -\infty$ as $t\to\infty$.

\begin{example}\label{ex3.4} \emph{Consider the $\log$-price $X_t$ governed by the equation $X_{t+1}=X_t+\varepsilon_{t+1}$, $t\in\mathbb{N}$, with $X_0=0$, where $\varepsilon_t$ are i.i.d. random variables in $\mathbb{R}$ with common distribution chosen such that $\mathbb{E}e^{-\varepsilon_1}>1$ and $\mathbb{E}\varepsilon_1 >0$. For example $\varepsilon_1\sim\mathcal{N}(1/4,1)$
will do. We identify the drift and volatility as $\mu\equiv 0$ and $\sigma\equiv 1$.}

 \emph{Choose the trading strategy $\pi_t\equiv 1$ for all $t$ and let $V_0=1$. Then we have $V_t:=\exp(\varepsilon_1+\cdots +\varepsilon_t)$ for all $t\ge 1$. Since $1/5<1/4=\mathbb{E}\varepsilon_1$, by Cram\'er's theorem (see e.g. \cite{DZ}), there are a constant $c>0$ and $t_0>0$ such that for all $t\ge t_0$, we have $\mathbb{P}(V_t\ge e^{t/5})\ge 1-e^{-ct}$. Hence there is $AEA$ with $GDPF$.}

 \emph{However, for $\alpha=-1$, we have, by independence,
$$
\mathbb{E}U(V_t)=\mathbb{E}(-V_t^{-1})=-\mathbb{E}\exp\{-(\varepsilon_1+\cdots +\varepsilon_t)\}=
-(\mathbb{E}e^{-\varepsilon_1})^t\to -\infty
$$ as $t\to\infty$.} \hfill $_{\blacksquare}$
\end{example}

 Finally we investigate what happens if a risk-averse agent produces expected utility tending to $0=U(\infty)$ exponentially fast as $t\to\infty$. It turns out that his/her strategy produces $AEA$ with $GDPF$. Indeed, following the footsteps of
Proposition 2.2 in \cite{FSr} we get the following result.

\begin{theorem}\label{t3.5} Consider the power utility $U(x)=-x^{\alpha}$ for some $\alpha<0$. 
Let $\pi_t$ be such that $|\mathbb{E}U(V_t^{\pi})|\le Ke^{-ct}$ for all large enough $t$, for some constants $c,K>0$. Then $\pi_t$ provides an $AEA$ with $GDPF$.
\end{theorem}

 {\bf Proof.} We may and will assume $K=1$. We need to find constants $b>0$, $c'>0$ such that 
 $\mathbb{P}(V_t^{\pi}\ge e^{bt})\ge 1-e^{-c't}$ for all large enough $t$. Choose $b>0$ such that $c+\alpha b>0$, then we have
\be\begin{array}{lll}
\mathbb{P}(V_t^{\pi}<e^{bt}) &=& \mathbb{P}\big(|U(V_t^{\pi})|> |U(e^{bt})|\big)\\
    &\le& \frac{\mathbb{E}|U(V_t^{\pi})|}{|U(e^{bt})|}\,\,\mbox{by Markov's inequality}.
\end{array}
\e

 But $\mathbb{E}|U(V_t^{\pi})|=|\mathbb{E}U(V_t^{\pi})|\le e^{-ct}$ and $|U(e^{bt})|=e^{\alpha bt}$ imply $\mathbb{P}(V_t^{\pi}<e^{bt})\le e^{-(c+\alpha b)t}$ for all $t$. 
Hence the result follows taking $c':=c+\alpha b$. \hfill $_{\blacksquare}$

 To conclude, if an economic agent with $HARA$ utility risk-aversion coefficient $\alpha<0$ achieves an expected utility that converges exponentially fast to $0$, then his/her strategy provides $AEA$ with $GDPF$, too. Conversely, under the stringent conditions of Section 2, one is able to construct strategies producing $AEA$ with $GDPF$ which also provide expected utilities tending to $0$ 
 exponentially fast {for $\alpha$ not
too negative} (i.e. for not too risk-averse investors).

\end{document}